\documentclass[a4paper]{article}

\usepackage[a4paper, top=3cm, bottom=2cm, left=3cm, right=3cm, marginparwidth=1.75cm]{geometry}

\usepackage{amsmath}
\usepackage{graphicx}
\usepackage[colorinlistoftodos]{todonotes}
\usepackage[colorlinks=true, allcolors=blue]{hyperref}
\usepackage{caption}

\usepackage[backend=bibtex, style=numeric, maxbibnames=99]{biblatex}
\renewbibmacro{in:}{}
\usepackage{amsthm, amssymb}
\usepackage{thm-restate}
\usepackage{mathtools}
\usepackage{bm}
\usepackage{tikz}
\usetikzlibrary{arrows, intersections, svg.path}
\usepackage{subcaption}

\newtheorem{theorem}{Theorem}[section]
\newtheorem{lemma}[theorem]{Lemma}
\newtheorem*{claim}{Claim}
\newtheorem{proposition}[theorem]{Proposition}

\theoremstyle{definition}
\newtheorem{remark}[theorem]{Remark}
\newtheorem{definition}[theorem]{Definition}

\newtheorem{question}[theorem]{Open question}

\newcommand{\sy}{\mathrm{SympEmb}}

\title{Noncontractible loops of symplectic embeddings between convex toric domains}
\author{Mihai Munteanu}

\begin{document}
\maketitle

\begin{abstract}
\noindent Given two $4$--dimensional ellipsoids whose symplectic sizes satisfy a specified inequality, we prove that a certain loop of symplectic embeddings between the two ellipsoids is noncontractible. The statement about symplectic ellipsoids is a particular case of a more general result. Given two convex toric domains whose first and second ECH capacities satisfy a specified inequality, we prove that a certain loop of symplectic embeddings between the two convex toric domains is noncontractible. We show how the constructed loops become contractible if the target domain becomes large enough. The proof involves studying certain moduli spaces of holomorphic cylinders in families of symplectic cobordisms arising from families of symplectic embeddings.
\end{abstract}

\section{Introduction}
\subsection{Previous results and a new result about ellipsoids}
Questions about symplectic embeddings of one symplectic manifold into another have always been one of the main study directions in symplectic geometry. The pioneering work of Gromov in \cite{gromov1985pseudo} introduced new methods that made it possible to answer many open questions about symplectic embeddings that had been until then unanswered. The survey by Schlenk, \cite{schlenk2018symplectic}, presents in detail the type of results one can prove about symplectic embeddings together with the tools used to prove such results. 

Most of the questions that have been answered (in the positive or the negative) concern the existence of symplectic embeddings of one symplectic manifold into another. For example, see \cite{mcduff1991blow},  \cite{mcduff2009symplectic}, \cite{mcduff1994symplectic}, and \cite{mcduff2012embedding} for symplectic embeddings involving $4$--dimensional ellipsoids, see \cite{choi2014symplectic}, \cite{cristofaro2014symplectic}, \cite{cristofaro2017symplectic}, and \cite{hutchings2016beyond} for symplectic embeddings involving more general $4$--dimensional symplectic manifolds, and also see \cite{hind2013ellipsoid}, \cite{guth2008symplectic}, and \cite{gutt2017symplectic} for results in higher dimensions. 

Another direction where significant progress has been made is the study of the connectivity of certain spaces of symplectic embeddings. In \cite{mcduff2009symplectic}, McDuff shows the connectivity of spaces of symplectic embeddings between $4$--dimensional ellipsoids, while in \cite{cristofaro2014symplectic}, Cristofaro--Gardiner extends this result to symplectic embeddings from concave toric domains to convex toric domains, both of which are subdomains of $\mathbb{R}^4$ whose definition we recall below in \S \ref{subsection:echcapacities}. In \cite{hind2013symplectic}, Hind proves the non-triviality of $\pi_0$ for spaces of symplectic embeddings involving certain $4$--dimensional polydisks, extending a result that was initially proved in \cite{floer1994applications}. In \cite{gutt2017symplectically}, the authors prove that certain spaces of symplectic embeddings involving more general $4$--dimensional symplectic manifolds are disconnected, while in \cite{mcduff1993camel}, the authors study the connectivity of symplectic embeddings into generalized ``camel" spaces in higher dimensions, extending results in \cite{eliashberg1991convex}.

Following yet another direction, in this paper we study the fundamental group of certain spaces of symplectic embeddings in $4$ dimensions. Let us first clarify the notation we will be using.  For real numbers $a$ and $b$ with  $0<a\leq b$, the set
\[ E(a,b):=\left\lbrace (z_1,z_2)\in\mathbb{C}^2\; \middle| \; \frac{\pi|z_1|^2}{a}+ \frac{\pi|z_2|^2}{b}\leq 1\right\rbrace \]
together with the restriction of the standard symplectic form from $\mathbb{R}^4$ is called a \textit{closed symplectic ellipsoid}, or more simply an \textit{ellipsoid}. Moreover, we define the symplectic ball $B^4(a):=E(a,a)$. Also, if $M$ and $N$ are symplectic manifolds, let $\sy(M,N)$ denote the space of symplectic embeddings of $M$ into $N$.

Here are a few results about the fundamental group of spaces of symplectic embeddings that motivated our work. The first result in this direction is an immediate consequence of the methods that Gromov introduced in \cite{gromov1985pseudo} in order to prove the nonsqueezing theorem.

\begin{theorem}[\cite{eliashberg1991convex}] Let $S$ be an embedded unknotted $2$--sphere in $(\mathbb{R}^4,\omega_{\mathrm{std}})$. Write $X_S=\mathbb{R}^4\setminus S$ and let $e:\sy(B^4(r), X_S)\to X_S$ be the evaluation map $f\mapsto f(0)$. Then the induced homomorphism $e_{*}:\pi_1(\sy(B^4(r), X_S))\to \pi_1(X_S)$ is surjective for $2\pi r^2 <\int_{S} \omega$ and trivial otherwise.
\end{theorem}

Another situation where the fundamental group of a space of symplectic embeddings can be computed is the following.

\begin{theorem}[\cite{hind2013symplectormophism}] If $\epsilon<1$ the space $\sy( B^4(\epsilon), B^4(1))$ deformation retracts to $U(2)$.
\end{theorem}

A more recent result that is closer in spirit to the results of this paper can be found in \cite{burkard2016first}, where the author constructs a loop $\{\phi_{\mu}\}_{\mu\in[0,1]}$ in $\sy(E(a,b)\sqcup E(a,b),B^4(R))$ and shows that if the positive real numbers $a$, $b$, and $R$ satisfy $\frac{a}{b} \notin \mathbb{Q}$, $2a<R<a+b$, and $b<2a$, then the constructed loop is noncontractible in  $\sy(E(a,b)\sqcup E(a,b), B^4(R))$.  Moreover, the loop becomes contractible if $R>a+b$. 

By contrast to \cite{burkard2016first}, we study symplectic embeddings whose domain is connected. More specifically, this paper is concerned with the study of restrictions of the loop of symplectic linear maps defined in (\ref{equation:loop}) below to certain domains in $\mathbb{R}^4$.

\begin{definition}
Let  $\lbrace \Phi_t \rbrace_{t\in[0,1]}\subset\mathrm{Sp}(4,\mathbb{R})$ denote the loop of symplectic linear maps
\begin{equation}
\label{equation:loop}
 \Phi_t(z_1,z_2) =
\begin{cases}
(e^{4\pi i t}z_1,z_2), & t\in\left[0,\frac{1}{2}\right] \\
(z_1,e^{-4\pi i t}z_2), & t\in\left[\frac{1}{2},1\right].
\end{cases}
\end{equation}
\end{definition}
The loop $\Phi_t$ is a concatenation of the $2\pi$ counterclockwise rotation in the $z_1$--plane followed by the $2\pi$ clockwise rotation in the $z_2$--plane. The loop $\{\Phi_t\}_{t\in[0,1]}$ is contractible in $\mathrm{Sp}(4,\mathbb{R})$, but it restricts to give some noncontractible loops of symplectic embeddings. For example:

\begin{theorem}
\label{theorem:ellipsoids}
Assume that $a<c<b<d$ and $c<2a$. Then, for $\Phi_t$ defined as in (\ref{equation:loop}), the loop of symplectic embeddings $\{\varphi_t=\Phi_t|_{E(a,b)}\}_{t\in[0,1]}$ is noncontractible in $\sy(E(a,b),E(c,d))$.
\end{theorem}

\begin{figure}[ht]
\centering
\begin{tikzpicture}[thick, > = stealth']
\begin{scope}
\draw[->] (-0.3,0) -- (4,0) coordinate[label = {below:$\pi|z_1|^2$}] (xmax);
\draw[->] (0,-0.3) -- (0,4) coordinate[label = {right:$\pi|z_2|^2$}] (ymax);
\draw[-] (0,2) -- (1,0) node[pos=1, below] {$a$} node[pos=0, left] {$b$} ;
\draw[-] (0,2.5) -- (1.5,0) node[pos=1, below] {$c$} node[pos=0, left] {$d$};
\end{scope}
\end{tikzpicture}
\captionsetup{justification=centering}
\caption{The loop $\{ \varphi_t\}_{t\in[0,1]}$ is noncontractible if \\ $a<c<b<d$ and $c<2a$.} 
\label{figure:ellipsoids}
\end{figure}
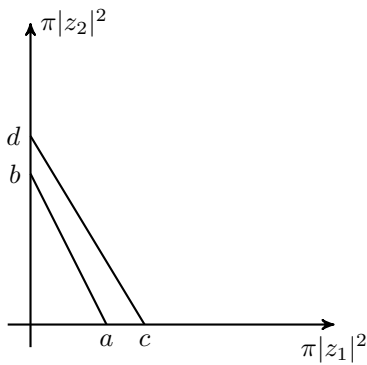

If $\max(a,b)\leq \min(c,d)$, then one can fit a ball between $E(a,b)$ and $E(c,d)$, meaning there exists $r>0$ such that $E(a,b) \subset B(r) \subset E(c,d)$, see Figure \ref{figure:ellipsoidsflipped}. Under this assumption, the loop $\{\varphi_t\}_{t\in[0,1]}$ is contractible. For a more general statement, see Proposition \ref{proposition:converse} below.

\begin{figure}[ht]
\centering
\begin{tikzpicture}[thick, > = stealth']
\begin{scope}
\draw[->] (-0.3,0) -- (4,0) coordinate[label = {below:$\pi|z_1|^2$}] (xmax);
\draw[->] (0,-0.3) -- (0,4) coordinate[label = {right:$\pi|z_2|^2$}] (ymax);
\draw[-] (0,2) -- (1,0) node[pos=1, below] {$a$} node[pos=0, left] {$b$} ;
\draw[-] (0,2.5) -- (3,0) node[pos=1, below] {$c$} node[pos=0, left] {$d$};
\draw[dash dot] (0,2.25) -- (2.25,0)  node[pos=1, below] {$r$} node[pos=0, left] {$r$} ;
\end{scope}
\end{tikzpicture}
\caption{The loop $\{ \varphi_t\}_{t\in[0,1]}$ is contractible if $\max(a,b)\leq \min(c,d)$.} \label{figure:ellipsoidsflipped}
\end{figure}
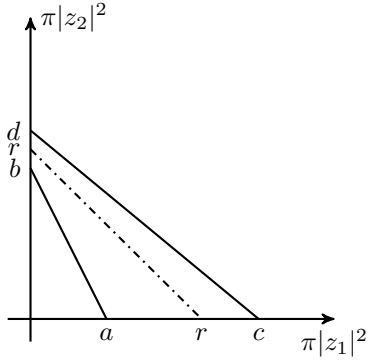

The method of proof we present in \S \ref{section:proof} does not answer whether the loop $\{\varphi_t\}_{t\in[0,1]}$ is contractible or not under the following assumption. 

\begin{question} 
\label{question:ellipsoids}
Assume $2a<c<b<d$. Is the loop $\{\varphi_t=\Phi_t|_{E(a,b)}\}_{t\in[0,1]}$ contractible in $\sy(E(a,b),E(c,d))$? 
\end{question}

\subsection{Main theorem}
\label{subsection:echcapacities}
We begin by recalling an important example of $4$--dimensional symplectic manifolds with boundary, in order to prepare for the statement of the main theorem. Given a domain $\Omega\subset \mathbb{R}_{\geq 0}^2$, we define the \textit{toric domain} 
\begin{equation}
\label{equation:toricdomain}
X_{\Omega}=\left\lbrace (z_1,z_2)\in\mathbb{C}^2\; \middle| \; \pi(|z_1|^2,|z_2|^2)\in\Omega\right\rbrace
\end{equation}
which, together with the restriction of the standard symplectic form $\omega_{\mathrm{std}}=dx_1\wedge dy_1 + dx_2\wedge dy_2$ on $\mathbb{C}^2$, is a symplectic manifold with boundary.
For example, if $\Omega$ is the triangle with vertices $(0,0)$, $(a,0)$ and $(0,b)$, then $X_{\Omega}$ is the ellipsoid $E(a,b)$ defined above, while if $\Omega$ is the rectangle with vertices $(0,0)$, $(a,0)$, $(0,b)$, and $(a,b)$, then $X_{\Omega}$ is the polydisk $P(a,b)=B^2(a)\times B^2(b)$. Note that we allow domains that have non-smooth boundary. The toric domains we work with in this paper have the following particular property.

\begin{definition} A \textit{convex toric domain} is a toric domain $X_{\Omega}$ defined by
\begin{equation} 
\label{equation:omega} \Omega = \left\lbrace (x,y)\in\mathbb{R}_{\geq 0}^2\; \middle| \; 0\leq x\leq a, \; 0 \leq y \leq f(x)\right\rbrace 
\end{equation}
such that its defining function $f:[0,a]\to \mathbb{R}_{\geq 0}$ is nonincreasing and concave.
\end{definition}

Even though we will not work with this type of domains in this paper, let us also recall that a \textit{concave toric domain} is a toric domain defined also by (\ref{equation:omega}) such that its defining function $f:[0,a]\to \mathbb{R}_{\geq 0}$ is nonincreasing, convex, and $f(a)=0$. For example, ellipsoids are the only toric domains that are both convex and concave, and polydisks are convex toric domains. We next explain how to compute the embedded contact homology (ECH) capacities of convex toric domains in order to state the main result of this paper.

Given a $4$--dimensional symplectic manifold $(X,\omega)$ with contact boundary $\partial X = Y$, its ECH capacities are a sequence of real numbers
\[0=c_0^{\mathrm{ECH}}(X,\omega) < c_1^{\mathrm{ECH}}(X, \omega) \leq \dots \leq \infty \]
constructed using a filtration by action of the ECH chain complex $ECC_{*}(Y, \lambda, J)$. The ECH capacities obstruct symplectic embeddings, meaning that if there exists a symplectic embedding $(X,\omega)\rightarrow (X',\omega')$ then $c_k(X,\omega) \leq c_k(X',\omega')$ for all $k\geq 0$. In particular, for the first and second ECH capacities of a convex toric domain, we can use the following explicit formulas, see \cite[Proposition 5.6]{hutchings2016beyond} for details.

\begin{proposition}
\label{proposition:echcapacities}
For a convex toric domain $X_{\Omega}$ with nice defining function $f:[0,a]\to\mathbb{R}_{\geq 0}$,
\begin{align*}
c_1^{\mathrm{ECH}}(X_\Omega) &= \min (a,f(0)) \text{ and} \\
c_2^{\mathrm{ECH}}(X_\Omega) &= \min (2a, x+f(x), 2f(0)),
\end{align*}
where $x\in(0,a)$ is the unique point where $f'(x)=-1$.

\end{proposition}

For the definition of a \textit{nice} defining function, see \S \ref{subsection:absolutegrading}. Every defining function can be perturbed to be nice. Having introduced all the ingredients, we are ready to state the main result of this paper.

\begin{theorem}
\label{theorem:maintheorem}
Let $X_{\Omega_1}$ and $X_{\Omega_2}$ be convex toric domains with defining functions $f_1:[0,a]\to\mathbb{R}_{\geq 0}$ and $f_2:[0,c]\to\mathbb{R}_{\geq 0}$, respectively. Assume that $X_{\Omega_1}\subset X_{\Omega_2}$, $a<c<f_1(0)<f_2(0)$, and $c_1^{\mathrm{ECH}}(X_{\Omega_2})< c_2^{\mathrm{ECH}}(X_{\Omega_1})$.  Then, for $\Phi_t$ defined as in (\ref{equation:loop}), the loop of symplectic embeddings $\{\varphi_t=\Phi_t|_{X_{\Omega_1}}\}_{t\in[0,1]}$ is noncontractible in $\sy(X_{\Omega_1},X_{\Omega_2})$.
\end{theorem}

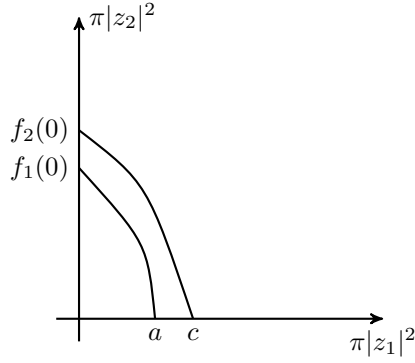
\begin{figure}[ht]
\centering
\begin{tikzpicture}[thick, > = stealth']
\begin{scope}
\draw[->] (-0.3,0) -- (4,0) coordinate[label = {below:$\pi|z_1|^2$}] (xmax);
\draw[->] (0,-0.3) -- (0,4) coordinate[label = {right:$\pi|z_2|^2$}] (ymax);
\draw[thick, black] (0,2) .. controls (0.9,1) .. (1,0)  node[pos=1, below] {$a$} node[pos=0, left] {$f_1(0)$};
\draw[thick, black] (0,2.5) .. controls (0.9,1.8) .. (1.5,0)  node[pos=1, below] {$c$} node[pos=0, left] {$f_2(0)$};
\end{scope}
\end{tikzpicture}
\captionsetup{justification=centering}
\caption{The loop $\{ \varphi_t\}_{t\in[0,1]}$ is noncontractible if \\
$X_{\Omega_1}\subset X_{\Omega_2}$, $a<c<f_1(0)<f_2(0)$,  and  $c_1^{\mathrm{ECH}}(X_{\Omega_2}) < c_2^{\mathrm{ECH}}(X_{\Omega_1})$.} \label{figure:maintheorem}
\end{figure}

\begin{remark} \mbox{}
\begin{enumerate}
\item[i.] By symmetry, Theorem \ref{theorem:maintheorem} also holds if we assume $f_1(0)<f_2(0)<a<c$ instead of $a<c<f_1(0)<f_2(0)$. See Figure \ref{figure:maintheorem} for an example where the bounds in the hypothesis of Theorem \ref{theorem:maintheorem} hold.
\item[ii.]  For $X_{\Omega_1}=E(a,b)$ and $X_{\Omega_2}=E(c,d)$ satisfying $a<c<b<d$, as in the hypothesis of Theorem \ref{theorem:ellipsoids}, we compute $c_1^{\mathrm{ECH}} (E(c,d))=\min(c,d)=c$ and $c_2^{\mathrm{ECH}} (E(a,b))=\min(2a, b)$. Hence, Theorem \ref{theorem:ellipsoids} is a special case of Theorem \ref{theorem:maintheorem}.
\end{enumerate}
\end{remark}

If the target $X_{\Omega_2}$ is large enough, the loop $\{ \varphi_t\}_{t\in[0,1]}$ becomes contractible, see Figure \ref{figure:converse}.
\begin{figure}[ht]
\centering
\begin{tikzpicture}[thick, > = stealth']
\begin{scope}
\draw[->] (-0.3,0) -- (4,0) coordinate[label = {below:$\pi|z_1|^2$}] (xmax);
\draw[->] (0,-0.3) -- (0,4) coordinate[label = {right:$\pi|z_2|^2$}] (ymax);
\draw[thick, black] (0,2) .. controls (0.9,1) .. (1,0)  node[pos=1, below] {$a$} node[pos=0, left] {$f_1(0)$};
\draw[thick, black] (0,3.5) .. controls (1.9,2.8) .. (2.5,0)  node[pos=1, below] {$c$} node[pos=0, left] {$f_2(0)$};
\draw[thick, dash dot] (0,2.3) -- (2.3,0) node[pos=1, below]{$r$} node[pos=0, left]{$r$};
\end{scope}
\end{tikzpicture}
\caption{If $X_{\Omega_1}\subset B(r) \subset X_{\Omega_2}$, the loop $\{ \varphi_t\}_{t\in[0,1]}$ is contractible.} \label{figure:converse}
\end{figure}
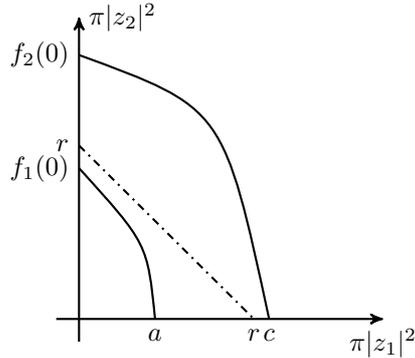
\begin{proposition} 
\label{proposition:converse}
Assume there exists $r>0$ such that $X_{\Omega_1} \subset B^4(r) \subset X_{\Omega_2}$. Then the loop $\{\varphi_t=\Phi_t|_{ X_{\Omega_1}}\}_{t\in[0,1]}$ is contractible in $\sy(X_{\Omega_1},X_{\Omega_2})$ . 
\end{proposition}

\begin{proof}

Since the loop $\{\Phi_t\}_{t\in[0,1]}$ is contractible in $U(2)$, there exists a homotopy of unitary maps $\{\Phi_z\}_{z\in \mathbb{D}}$  contracting it, where $\mathbb{D}$ denotes the closed unit disk. For each $z\in \mathbb{D}$, the operator norm of $\Phi_z \in U(2)$ is $||\Phi_z || = 1$, and hence $\mathrm{im} \left(\Phi_z |_{X_{\Omega_1}}\right) \subset B(r) \subset X_{\Omega_2}$. So the $2$--parameter family of restrictions $\{ \Phi_z|_{X_{\Omega_1}} \}_{z\in\mathbb{D}}$ is contained in $\sy(X_{\Omega_1},X_{\Omega_2})$ and provides a homotopy from $\{\varphi_t\}_{t\in[0,1]}$ to the constant loop.
\end{proof}

\subsection{Strategy of proof and the organization of the paper}
We use the following strategy to prove Theorem \ref{theorem:maintheorem}. For each symplectic embedding $\varphi: X_{\Omega_1}\to X_{\Omega_2}$, we add to the compact symplectic cobordism ($X_{\Omega_2}\setminus\mathrm{int}(\varphi (X_{\Omega_1})),\omega_{\mathrm{std}})$, a positive cylindrical end at $\partial X_{\Omega_2}$ and a negative cylindrical end at $\partial X_{\Omega_1}$, in order to construct the completed symplectic cobordism $\widehat{W}_{\varphi}=(-\infty,0]\times \partial X_{\Omega_1} \cup (X_{\Omega_2} \setminus \mathrm{int}\; \varphi (X_{\Omega_1}))\cup [0,\infty) \times \partial X_{\Omega_2}$. After choosing an almost complex structure $J$ that is compatible with the cobordism structure on $\widehat{W}_{\varphi}$,  we define the moduli space $\mathcal{M}_{J}(\varphi)$ which consists of $J$--holomorphic cylinders in $\widehat{W}_{\varphi}$ that have a positive end at the shortest Reeb orbit on $\partial X_{\Omega_2}$ and a negative end at the shortest Reeb orbit on $\partial X_{\Omega_1}$. 

Using automatic transversality together with a compactness argument which works under the hypothesis of Theorem \ref{theorem:maintheorem}, we show that for each $\varphi\in\sy( X_{\Omega_1}, X_{\Omega_2})$ and for each compatible almost complex structure $J$, the moduli space $\mathcal{M}_{J}(\varphi)$ is a finite set. We directly construct an almost complex structure $\widehat{J}$ and a $\widehat{J}$--holomorphic cylinder with the right asymptotics, to show that $\mathcal{M}_{\widehat{J}}(\varphi_0)$ is nonempty for the restriction of the inclusion map $\varphi_0$ and the particular choice of $\widehat{J}$. We describe the cylinders near their asymptotic ends to prove that, whenever nonempty,  $\mathcal{M}_{J}(\varphi)$ contains a unique $J$--holomorphic cylinder. 

We complete the proof using an argument by contradiction. We assume the loop $\{\varphi_t\}_{t\in[0,1]}$ is contractible by the homotopy $\{\varphi_z\}_{z\in \mathbb{D}}$, $\varphi_z \in \sy( X_{\Omega_1}, X_{\Omega_2})$ for each $z\in\mathbb{D}$. We choose a $2$--parameter family of almost complex structures $\mathfrak{J}=\{J_z\}_{z\in \mathbb{D}}$ so that $J_z$ is compatible with the cobordism structure on $\widehat{W}_{\varphi_z}$ and $J_z=\widehat{J}$ for all $z\in\partial \mathbb{D}$. We define the universal moduli space $\mathcal{M}_{\mathfrak{J}}=\sqcup_{z\in\mathbb{D}} \mathcal{M}_{J_z}(\varphi_z)$ and, using parametric transversality for generic families of almost complex structures, we show that, for a generic choice of $\mathfrak{J}$ as above, the moduli space $\mathcal{M}_{\mathfrak{J}}$ is a $2$--dimensional manifold. Assuming the bounds in the hypothesis of Theorem \ref{theorem:maintheorem}, we conclude using SFT compactness and the description of each $\mathcal{M}_{J_z}(\varphi_z)$ that $\mathcal{M}_{\mathfrak{J}}$ is homeomorphic to the closed disk $\mathbb{D}$. 

For the final details, we fix a parametrization of the shortest Reeb orbits on $\partial X_{\Omega_2}$ together with a point $p$ on the same Reeb orbit. For each $\varphi_z$, we trace, on the unique cylinder $[u_z]\in \mathcal{M}_{J_z}(\varphi)$, the vertical ray that is asymptotic to $p$ at $\infty$ and record the point $p_z$ where it lands at $-\infty$ on the shortest Reeb orbit on $\partial X_{\Omega_1}$. We then study the composition of maps
\[
\begin{array}{ccccccc}
S^1 & \to & \sy(X_{\Omega_1},X_{\Omega_2})  & \to & \mathcal{M}_{\mathfrak{J}} & \to & S^1 \\
t & \mapsto & \varphi_t=\varphi_z & \mapsto & (z,[u_z]) & \mapsto & p_z.
\end{array}
\]
and show that this circle map has degree $-1$. This provides the contradiction we are looking for, since we previously showed that $\mathcal{M}_{\mathfrak{J}}$ is homeomorphic to the closed disk $\mathbb{D}$.

The paper is divided in sections as follows. In \S \ref{section:orbits}, we classify the embedded Reeb orbits on the boundary of a convex toric domain. We make use of this classification, together with an automatic transversality argument, to prove the compactness of the moduli space $\mathcal{M}_{J}(\varphi)$ in \S \ref{section:rullingout}. We also use the classification in \S \ref{section:orbits} to show the compactness of the moduli space $\mathcal{M}_{\mathfrak{J}}$ in \S \ref{subsection:proof}. Finally, \S \ref{subsection:nonemptiness} contains the argument for the existence of $J$--holomorphic cylinders with the right asymptotics, \S \ref{subsection:uniqueness} contains the argument for the uniqueness of $J$--holomorphic cylinders in $\mathcal{M}_{J}(\varphi)$, and \S \ref{subsection:proof} presents the details behind the construction of the circle map above, in order to complete the proof.

\textbf{Acknowledgements.} I would like to thank my advisor, Michael Hutchings, for all the help and ideas he shared with me. I would also like to thank Chris Wendl for clarifying some of my mathematical confusions during my visit at Humboldt--Universit\"at zu Berlin. Finally, I would like to thank my friends, Julian Chaidez and Chris Gerig, for the many helpful conversations we had.
\section{Reeb dynamics and the ECH index}
\label{section:orbits}
\subsection{Geometric setup}
Let $(Y,\xi)$ be a closed $3$--dimensional contact manifold with contact form $\lambda$, i.e. $\xi = \ker \lambda$. The \textit{Reeb vector field} $R$ corresponding to $\lambda$ is uniquely defined as the vector field satisfying $d\lambda(R,\cdot)=0$ and $\lambda(R)=0$. A \textit{Reeb orbit} is a map $\gamma:\mathbb{R}/T\mathbb{Z}\to Y$ for some $T>0$, modulo translations of the domain, such that $\gamma^{\prime}(t)=R(\gamma(t))$. The \textit{action} of a Reeb orbit $\gamma$ is defined by $\mathcal{A}(\gamma)=\int_{S^1} \gamma^{*} \lambda$ and is also equal to the period of $\gamma$.

For a fixed Reeb orbit $\gamma$, the linearization of the Reeb flow of $R$ induces a symplectic linear map $P_{\gamma}: (\xi_{\gamma(0)},d\lambda)\to(\xi_{\gamma(0)},d\lambda)$, called the \textit{linearized return map}.
A Reeb orbit $\gamma:\mathbb{R}/T\mathbb{Z}$ is called \textit{nondegenerate} if its linearized return map $P_{\gamma}$ does not have $1$ as an eigenvalue. We call $\gamma$ \textit{elliptic} if the eigenvalues of $P_{\gamma}$ are complex conjugate on the unit circle,  \textit{positive hyperbolic} if the eigenvalues of $P_{\gamma}$ are real and positive, and  \textit{negative hyperbolic} if the eigenvalues of $P_{\gamma}$ are real and negative. A contact form $\lambda$ is called \textit{nondegenerate} if all its Reeb orbits are nondegenerate.
\subsection{Reeb dynamics on $\partial X_{\Omega}$}
\label{subsection:reebdynamics}
In this section we compute the Reeb dynamics on the boundary of convex toric domain. Recall that a convex toric domain $X_{\Omega}\subset \mathbb{R}^4$ is defined by (\ref{equation:toricdomain}), with defining set $\Omega$ given by (\ref{equation:omega}). Similarly to the computations in \cite[\S 4.3]{hutchings2014lecture}, we choose scaled polar coordinates $(z_1,z_2)=(\sqrt{r_1/\pi}e^{i\theta_1},\sqrt{r_2/\pi}e^{i\theta_2})$ on $\mathbb{C}^2$  to obtain
\[ \omega_{\mathrm{std}} = \frac{1}{2\pi} \left(dr_1\wedge d\theta_1 + dr_2 \wedge d\theta_2 \right).\] 

The radial vector field 
\[ \rho = r_1 \frac{\partial}{\partial r_1} + r_2 \frac{\partial}{\partial r_2} \]
is a Liouville vector field for $\omega_{\mathrm{std}}$ defined on all $\mathbb{R}^4$. The boundary of the toric domain $\partial X_{\Omega}$ is transverse to $\rho$ and so
\[ \lambda_{\mathrm{std}} = \iota_{\rho} \omega_{\mathrm{std}} = \frac{1}{2\pi}\left(r_1d\theta_1 + r_2 d\theta_2\right)\]
restricts to a contact form on $\partial X_{\Omega}$.  The Reeb vector field $R$ corresponding to $\lambda_{\mathrm{std}}$ has the following expression. In the two coordinate planes, $R$ is given by 
\[ R = 
\begin{cases}
\frac{2\pi}{a}\frac{\partial}{\partial \theta_1} & \text{ if } z_2=0 \\
\frac{2\pi}{f(0)}\frac{\partial}{\partial \theta_2} & \text{ if } z_1=0.
\end{cases}
\]
While if $\pi(|z_1|^2,|z_2|^2)=(r_1,r_2)=(x,f(x))$ for some $x\in(0,a)$ with $f'(x)=\tan \phi$, $\phi \in [-\pi/2,0]$, then
\[ R = \frac{2\pi}{-x \sin \phi  + f(x) \cos \phi }\left(-\sin\phi \frac{\partial}{\partial \theta_1}+ \cos \phi \frac{\partial}{\partial \theta_2}\right).\]

The embedded Reeb orbits or $\lambda_{\mathrm{std}}|_{\partial X_{\Omega}}$ are classified as follows:
\begin{itemize}
\item The circle $e_{0,1}=\partial X_{\Omega} \cap \{z_2=0\}$ is an embedded elliptic Reeb orbit with action $\mathcal{A}(e_{0,1})=a$.
\item The circle $e_{1,0}=\partial X_{\Omega} \cap \{z_1=0\}$ is an embedded elliptic Reeb orbit with action $\mathcal{A}(e_{1,0})=f(0)$.
\item For each $x\in(0,a)$ with $f'(x)\in\mathbb{Q}$, the torus
\[ \{ z\in\partial X_{\Omega} | \pi(|z_1|^2,|z_2|^2)=(x,f(x)) \}\]
is foliated by a Morse-Bott circle of Reeb orbits. If $f'(x)=-\frac{p}{q}$ with $p, q$ relatively prime positive integers, then we call this torus $T_{p,q}$ and we compute that each orbit in this family has action $\mathcal{A} = q x + p f(x) $.
\end{itemize}

\begin{remark}
\label{remark:perturbcontactform}
The existence of Morse-Bott circles of Reeb orbits implies that the contact form $\lambda_{\mathrm{std}}|_{\partial X_{\Omega}}$ is degenerate. We need to perturb it in order to make it nondegenerate since the nondegeneracy allows the study of $J$--holomorphic curves with cylindrical ends asymptotic to Reeb orbits. 

For each $\epsilon>0$, we can perturb $\lambda_{\mathrm{std}}|_{\partial X_{\Omega}}$ to a nondegenerate $\lambda = h\lambda_{\mathrm{std}}|_{\partial X_{\Omega}}$, where $||h-1||_{C^0}<\epsilon$, so that each Morse-Bott family $T_{p,q}$ that has action $\mathcal{A} < 1/\epsilon$ becomes two embedded Reeb orbits of approximately the same action, more specifically an elliptic orbit $e_{p,q}$ and a hyperbolic orbit $h_{p,q}$. Moreover, no Reeb orbits of action $\mathcal{A} < 1/\epsilon$ are created and the Reeb orbits $e_{0,1}$ and $e_{1,0}$ are unaffected.

Such a perturbation of the contact form is equivalent to a perturbation of the hypersurface $\partial X_{\Omega}$ on which the restriction of $\lambda_{\mathrm{std}}$ becomes nondegenerate.
\end{remark}

\subsection{ECH index}
\label{subsection:ECHindex}
Embedded contact homology (ECH) is an invariant for $3$--dimensional contact manifolds due to Hutchings. See \cite{hutchings2014lecture} for a detailed account of history, motivation, construction, and applications of ECH. We give a brief overview of the definition of ECH following the notation from \cite{hutchings2009embedded}.

Let $(Y, \lambda)$ be a contact $3$--dimensional manifold with nondegenerate contact form $\lambda$. Given a convex toric domain $X_{\Omega}$, the boundary $\partial X_{\Omega}$ together with a perturbation of $\lambda_{\mathrm{std}}|_{\partial X_{\Omega}}$, as in Remark \ref{remark:perturbcontactform}, is such a contact manifold. 

An \textit{orbit set} is a finite set of pairs $\alpha=\{ (\alpha_i, m_i)\}$, where $\alpha_i$ are distinct embedded Reeb orbits and $m_i$ are positive integers. We will also use the multiplicative notation $\alpha = \prod \alpha_i^{m_i}$ for an orbit set $\alpha=\{ (\alpha_i, m_i)\}$. Denote by $[\alpha]$ the sum $\sum_i m_i[\alpha_i] \in H_1(Y)$ and define the action of $\alpha$ by $\mathcal{A}(\alpha)=\sum_i m_i\mathcal{A}(\alpha_i)$. If $\alpha=\{ (\alpha_i, m_i)\}$ and $\beta=\{ (\beta_j, n_j)\}$ are two orbit sets with $[\alpha]=[\beta]\in H_1(Y)$, then define $H_2(Y,\alpha,\beta)$ to be the set of relative homology classes of $2$--chains $Z$ such that $\partial Z = \sum m_i \alpha_i - \sum n_j\beta_j$. Note that 
$H_2(Y,\alpha,\beta)$ is an affine space over $H_2(Y)$.

Given a $Z\in H_2(Y, \alpha, \beta)$, define the \textit{ECH index} of $Z$  by the formula
\begin{equation} 
\label{equation:echindex}
I(\alpha,\beta,Z) = c_{\tau}(Z)+Q_{\tau}(Z)+CZ^{I}_{\tau}(\alpha)-CZ^{I}_{\tau}(\beta)
\end{equation}
where $\tau$ is a choice of symplectic trivializations of $\xi$ over the Reeb orbits  $\alpha_i$ and $\beta_j$, $c_{\tau}(Z)=c_1(\xi|_Z,\tau)$ denotes the relative first Chern class (see \cite[\S 2.5]{hutchings2009embedded}), $Q_{\tau}(Z)$ denotes the relative self-intersection number (see \cite[\S 2.7]{hutchings2009embedded}), and 
\[CZ^{I}_{\tau}(\alpha) = \sum_i \sum_{k=1}^{m_i} CZ_{\tau}(\alpha^k_i),\]
where $CZ_{\tau}(\gamma)$ is the Conley--Zehnder index with respect to $\tau$ of the orbit $\gamma$ (see \cite[\S 2.3]{hutchings2009embedded}). 

The ECH index does not depend on the choice of symplectic trivialization.  The definition of the ECH index $I$ can be extended to symplectic cobordisms by generalizing the definitions of the relative first Chern class and of the self intersection number (see \cite[\S 4.2]{hutchings2009embedded}). 

If $Z\in H_2(Y, \alpha, \beta)$ and $W\in H_2(Y, \beta, \gamma)$, then $I(Z+W)=I(Z)+I(W)$. In the particular case of starshaped hypersurfaces in $\mathbb{R}^4$, this implies there is an absolute $\mathbb{Z}$ grading on orbit sets as follows. Since $H_2(Y)=H_2(S^3)=0$, for every pair of orbit sets $\alpha$ and $\beta$ there is an unique class $Z\in H_2(Y, \alpha, \beta)$. Define $I(\emptyset)=0$ for the empty orbit set and set
\[ I(\alpha):=I(\alpha,\emptyset,Z)\in\mathbb{Z},\]
where $Z$ is the unique element of $H_2(Y,\alpha,\emptyset)$. Also, let $c_{\tau}(\alpha):=c_{\tau}(Z)$ and $Q_{\tau}(\alpha):=Q_{\tau}(Z)$.

\subsection{Absolute grading on $\partial X_{\Omega}$}
\label{subsection:absolutegrading}
Following the details in \cite[\S 5]{hutchings2016beyond}, we recall the classification of the orbit sets on the boundary of a convex toric domain $X_{\Omega}$ that have ECH index $I\leq4$. 

Similarly to \cite[Lemma 5.4]{hutchings2016beyond}, we first perform a perturbation of the geometry of $\partial X_\Omega$  (see Figure \ref{figure:generalposition}). This means we can assume, without loss of generality, that the function $f:[0,a]\to\mathbb{R}_{\geq 0}$ defining $\Omega$ is \textit{nice}, meaning that $f$ satisfies the following properties:
\begin{itemize}
\item $f$ is smooth,
\item $f'(0)$ is irrational and is approximately $0$,
\item $f'(a)$ is irrational and is very large, close to $-\infty$,
\item $f''(x)<0$ except for $x$ in small connected neighborhoods of $0$ and $a$.
\end{itemize}

\begin{figure}[h]
\centering
\begin{subfigure}{0.4\textwidth}
\begin{tikzpicture}[thick, > = stealth']
\begin{scope}
\draw[->] (-0.3,0) -- (4,0) coordinate[label = {below:$\pi|z_1|^2$}] (xmax);
\draw[->] (0,-0.3) -- (0,4) coordinate[label = {right:$\pi|z_2|^2$}] (ymax);
\draw [black] plot [smooth] coordinates { (0,2.8) (1,2.5)  (2,2) (2.5,1) (2.8,0)};
\filldraw[black] (2.8,0) circle (0pt) node[anchor=north] {$a$};
\filldraw[black] (0,2.8) circle (0pt) node[anchor=east] {$f(0)$};
\end{scope}
\end{tikzpicture}
\caption{Original toric domain}
\end{subfigure}
$\longrightarrow$
\begin{subfigure}{0.4\textwidth}
\begin{tikzpicture}[thick, > = stealth']
\begin{scope}
\draw[->] (-0.3,0) -- (4,0) coordinate[label = {below:$\pi|z_1|^2$}] (xmax);
\draw[->] (0,-0.3) -- (0,4) coordinate[label = {right:$\pi|z_2|^2$}] (ymax);
\draw [black!50] plot [smooth] coordinates { (0,2.8) (1,2.5)  (2,2) (2.5,1) (2.8,0)};
\draw [black] plot [smooth] coordinates { (0,2.63) (1,2.5) (2,2) (2.5,1) (2.63,0)};
\filldraw[black] (2.6,0) circle (0pt) node[anchor=north] {$a$};
\filldraw[black] (0,2.6) circle (0pt) node[anchor=east] {$f(0)$};
\end{scope}
\end{tikzpicture}
\caption{Defining function perturbed to be nice}
\end{subfigure}
\caption{Perturbating $X_{\Omega}$ to a nice position} \label{figure:generalposition}
\end{figure}
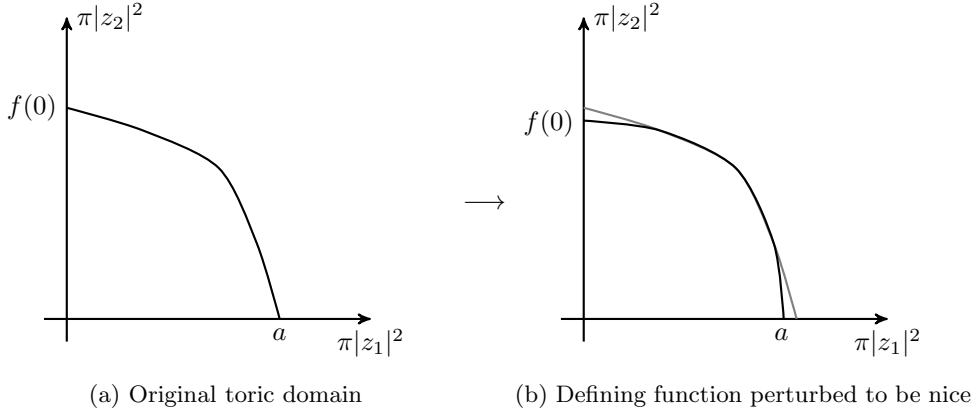

\begin{lemma}[{\cite[Example 1.12]{hutchings2016beyond}}]
\label{lemma:indexclassification}
Let $X_{\Omega}$ be a convex toric domain defined by a nice function $f$. Let  $\lambda$ be a nondegenerate contact structure obtained by perturbing $\lambda_{\mathrm{std}}|_{\partial X_{\Omega}}$ up to sufficiently large action. Then the orbit sets with ECH index $I\leq 4$ are classified as follows.
\begin{itemize}
\item $I=0$: $\emptyset$.
\item $I=1$: no orbit sets.
\item $I=2$: $e_{0,1}$ and $e_{1,0}$.
\item $I=3$: $h_{1,1}$.
\item $I=4$: $e_{0,1}^2$,  $e_{1,1}$, and $e_{1,0}^2$.
\end{itemize}
\end{lemma}

In general, the classification of orbit set generators, up to larger ECH index and action, provides a combinatorial model to compute the sequence of ECH capacities of a convex toric domain using the following formula.
\begin{lemma}[{\cite[Lemma 5.6]{hutchings2016beyond}}]
\label{lemma:echcapacities}
For a convex toric domain $X_{\Omega}$ and a nonnegative integer $k$,
\[c_k^{\mathrm{ECH}}(X_{\Omega})= \mathrm{min}\{\mathcal{A}(\alpha) \; | \; I(\alpha)=2k\}.\]
\end{lemma}
In particular, the equalities claimed in Proposition \ref{proposition:echcapacities} hold. Moreover, one can deduce the following lemma which we will use to rule out breaking.

\begin{lemma}
\label{lemma:higherindexorbits}
For a convex toric domain $X_{\Omega}$, orbit sets $\alpha$ with ECH index $I(\alpha)\geq 5$ have action $\mathcal{A}(\alpha)\geq c_2^{\mathrm{ECH}}(X_{\Omega})$. 
\end{lemma}

\section{Ruling out breaking}
\label{section:rullingout}
\subsection{Completed symplectic cobordisms}
Let $(Y_{\pm},\lambda_{\pm})$ be closed contact $3$--dimensional manifolds. A \textit{compact symplectic cobordism} from $(Y_{+},\lambda_{+})$ to $(Y_{-},\lambda_{-})$ is a compact symplectic manifold $(W,\omega)$ with boundary $\partial W = -Y_{-}\sqcup Y_{+}$ such that $\omega|_{Y_{\pm}}=d\lambda_{\pm}$.  

Given a compact symplectic cobordism $(W,\omega)$, one can find neighborhoods $N_{-}$ of $Y_{-}$ and $N_{+}$ of $Y_{+}$ in $W$, and symplectomorphisms
\[ (N_{-}, \omega) \to ([0,\epsilon)\times Y_{-}, d(e^s\lambda_{-}))\] 
and
\[ (N_{+}, \omega) \to ((-\epsilon,0]\times Y_{+}, d(e^s\lambda_{+})),\]
where $s$ denotes the coordinate on $[0,\epsilon)$ and $(-\epsilon,0]$.
Using these identifications, we can complete the compact symplectic cobordism $(W,\omega)$ by adding cylindrical ends $(-\infty,0]\times  Y_{-}$ and $[0,\infty) \times Y_{+}$ to obtain the \textit{completed symplectic cobordism}
\[ \widehat{W} = [0,\infty)\times Y_{+} \cup_{Y_{+}} W \cup_{Y_{-}} (-\infty, 0]\times Y_{-}.\]

In accordance with \cite{bourgeois2003compactness}, we restrict the class of almost complex structures on a completed cobordism $\widehat{W}$ as follows. An almost complex structure $J$ on a completed symplectic cobordism $\widehat{W}$ as above is called \textit{compatible} (in \cite{bourgeois2003compactness}, the authors use the term \textit{adjusted}) if:
\begin{enumerate}
\item[$\cdot$] On $[0,\infty)\times Y_{+}$ and $(-\infty, 0]\times Y_{-}$, the almost complex structure $J$ is $\mathbb{R}$--invariant, maps $\partial_s$ (the $\mathbb{R}$ direction) to $R_{\lambda_{\pm}}$, and maps $\xi_{\pm}$ to itself compatibly with $d\lambda_{\pm}$.
\item[$\cdot$] On the compact symplectic cobordism $W$, the almost complex structure $J$ is tamed by $\omega$.
\end{enumerate}
Call $\mathcal{J}(\widehat{W})$ the set of all such compatible almost complex structures on $\widehat{W}$.

Choose a compatible almost complex structure $J\in\mathcal{J}(\widehat{W})$ on $\widehat{W}$ and a let $(\Sigma, j)$ be a compact Riemann surface. We will consider curves $u:(\dot{\Sigma}=\Sigma \setminus \{x_1, \dots, x_k, y_1,\dots, y_l\},j)\to (W,J)$ that are $J$--holomorphic, i.e. $du\circ j = J \circ du$, and have $k$ positive ends at $\Gamma^{+} = (\gamma^{+}_1,\dots,\gamma^{+}_k)$ corresponding to the punctures $(x_1,\dots, x_k)$, and $l$ negative ends at $\Gamma^{-}=(\gamma^{-}_1,\dots,\gamma^{-}_l)$ corresponding to the punctures $(y_1,\dots, y_l)$.  Denote by $\mathcal{M}_J(\Gamma^{+}, \Gamma^{-})$ the space of such $J$--holomorphic curves $u$ modulo reparametrizations of the domain $\dot{\Sigma}$.

Recall that a \textit{positive end} of $u$ at $\gamma$ means a puncture, near which, $u$ is asymptotic to $\mathbb{R}\times \gamma$. More specifically, that means there is a choice of coordinates $(s,t)\in [0,\infty)\times\mathbb{R}/T\mathbb{Z}$ on a neighborhood of the puncture, with $j(\partial_s)=\partial_t$ and such that $\lim_{s\to\infty}\pi_{\mathbb{R}}(u(s,t))= \infty$ and $\lim_{s\to\infty}\pi_{Y_{+}}(u(s,\cdot))= \gamma$. Similarly, a \textit{negative end} at $\gamma$ is a puncture, near which, $u$ is asymptotic to $\mathbb{R}\times \gamma$. More specifically, that means there is a choice of coordinates $(s,t)\in (-\infty,0]\times\mathbb{R}/T\mathbb{Z}$ on a neighborhood of the puncture, with $j(\partial_s)=\partial_t$ and such that $\lim_{s\to\infty}\pi_{\mathbb{R}}(u(s,t))= \infty$ and $\lim_{s\to\infty}\pi_{Y_{-}}(u(s,\cdot))= \gamma$. 

Given a $J$--holomorphic curve $u$ as above, define the \textit{Fredholm index} of $u$ by
\begin{equation}
\label{equation:Fredholmindex}
\mathrm{ind}(u) = -\chi(u)+2c_{\tau}(u)+\sum_{i=1}^{k} CZ_{\tau}(\gamma^{+}_i)-\sum_{j=1}^l CZ_{\tau}(\gamma^{-}_j), 
\end{equation}
where $\tau$ is a trivialization of $\xi$ over $\gamma^{\pm}_i$ that is symplectic with respect to $d\lambda$, $\chi(u)$ is the Euler characteristic of $\dot{\Sigma}$, $c_{\tau}(u):=c_1(u^{*}\xi,\tau)$ denotes the relative first Chern class, and $CZ_{\tau}(\gamma^{\pm}_i)$ is the Conley--Zehnder index with respect to $\tau$, as before. 
The significance of the Fredholm index is that for a generic choice of compatible almost complex structure $J$ and for a somewhere--injective $J$--holomorphic curve $u$, the moduli space $\mathcal{M}_J(\Gamma^{+}, \Gamma^{-})$ is a manifold of dimension $\mathrm{ind}(u)$ near $u$. See \cite[\S 6]{wendl2016lectures} for more details.

\subsection{Moduli spaces}

Let $X_{\Omega_1}$ and $X_{\Omega_2}$ be two convex toric domains defined by nice functions $f_1:[0,a]\to \mathbb{R}_{\geq 0}$ and $f_2:[0,a]\to \mathbb{R}_{\geq 0}$, respectively. Also, let $\varphi:X_{\Omega_1} \to X_{\Omega_2}$ be a symplectic embedding. The manifold $W_{\varphi} :=X_{\Omega_2}\setminus \mathrm{int} (X_{\Omega_1})$ is a compact symplectic cobordism from $(\partial X_{\Omega_2},\lambda_{\mathrm{std}}|_{\partial X_{\Omega_2}})$ to $(\partial X_{\Omega_1},\lambda_{\mathrm{std}}|_{\partial X_{\Omega_1}})$, where $\lambda_{\mathrm{std}}$ denotes the standard Liouville form on $\mathbb{R}^4$. 

Following the explanation in Remark \ref{remark:perturbcontactform}, perturb the boundary components $\partial X_{\Omega_1}$ and $\partial X_{\Omega_2}$ of $W_{\varphi}$ in such a way that the Liouville form $\lambda_{\mathrm{std}}$ restricts to nondegenerate contact forms $\lambda_1$ and $\lambda_2$  on $\partial X_{\Omega_1}$ and $\partial X_{\Omega_2}$, respectively.  Add cylindrical ends to $W_{\varphi}$ and call $\widehat{W}_{\varphi}$ the completed symplectic cobordism.

To clean up notation, call $\gamma_a$ the $e_{0,1}$ embedded Reeb orbit on $\partial X_{\Omega_1}$, and call $\gamma_c$ the $e_{0,1}$ embedded Reeb orbit on $\partial X_{\Omega_2}$. Recall that $\mathcal{A}(\gamma_a)=a$ and $\mathcal{A}(\gamma_c)=c$.

For a given almost complex structure $J\in\mathcal{J}(\widehat{W}_{\varphi})$, define $\mathcal{M}_J({\varphi})$ to be the moduli space of $J$--holomorphic cylinders $u:(\mathbb{R}\times S^1,j)\to (\widehat{W}_{\varphi},J)$ such that $u$ has a positive end at $\gamma_c$ and a negative end at $\gamma_a$, modulo translation and rotations of the domain $\mathbb{R}\times S^1$. 

All such $J$--holomorphic cylinders have Fredholm index $\mathrm{ind}(u)=0$ and the automatic transversality result in Lemma \ref{lemma:automatictransversality} below implies that $\mathcal{M}_J({\varphi})$ is a $0$--dimensional manifold for any choice of $J$. Moreover, $\mathcal{M}_J({\varphi})$ can be compactified with broken holomorphic curves using the SFT compactness theorem, \cite[Theorem 10.2]{bourgeois2003compactness}, since all the $J$--holomorphic cylinders in $\mathcal{M}_J({\varphi})$ have the same asymptotics.

\subsection{Automatic transversality}
A much more general automatic transversality result than the one we need to use is proven by Wendl in \cite{wendl2008automatic}. In the language employed in this paper, the particular case that we need to use is stated as follows. See also \cite[Lemma 4.1]{hutchings2014cylindrical} for a very similar statement and proof in the case of symplectizations.

\begin{lemma} 
\label{lemma:automatictransversality}
Let $\widehat{W}$ be a completed symplectic cobordism and let $u:\dot{\Sigma}\to \widehat{W}$ be an immersed $J$--holomorphic curve that has asymptotic ends to Reeb orbits. Let $N$ denote the normal bundle to $u$ in $X$ and 
\[ D_u : L^2(\Sigma, N) \to L^2(\Sigma, T^{0,1}\mathbb{C} \otimes N) \]
denote the normal linearized operator of $u$. Also let $h^{+}(u)$ denote the number of ends of $u$ at positive hyperbolic orbits. If 
\[2g(\Sigma)-2+h^{+}(u)<\mathrm{ind}(u),\]
then $D_u$ is surjective, i.e. the moduli space of $J$--holomorphic curves near $u$ is a manifold that is cut out transversely and has dimension $\mathrm{ind}(u)$. 
\end{lemma}

Note that there are no genericity assumptions on the almost complex structure $J$ in Lemma \ref{lemma:automatictransversality}. Also, the result applies to the $J$--holomorphic cylinders in $\mathcal{M}_J({\varphi})$ since they have ends only at elliptic Reeb orbits and the adjunction formula introduced below in (\ref{equation:adjunctionformula}) implies that they are embedded. Hence $\mathcal{M}_J({\varphi})$ is cut out transversely, for any choice of compatible almost complex structure $J$. 

\subsection{Ruling out breaking}

In this section, we study the possible boundary of the union $\sqcup_{J\in\mathfrak{J}}\mathcal{M}_J({\varphi})$, where $\mathfrak{J}$ is a smooth parametrized family of compatible almost complex structures. We prove that, assuming the bounds in the hypothesis of Theorem \ref{theorem:maintheorem}, a sequence of cylinders in $\sqcup_{J\in\mathfrak{J}}\mathcal{M}_J({\varphi})$ cannot converge to a broken holomorphic building with multiple levels.

\begin{proposition} 
\label{proposition:nobadbreaking}
Assume $X_{\Omega_1}$ and $X_{\Omega_2}$ are convex toric domains satisfying the bounds in the hypothesis of Theorem \ref{theorem:maintheorem}. Let $\{\varphi_i \in \sy(X_{\Omega_1},X_{\Omega_2})\}_{i\geq 1}$ be a sequence of symplectic embeddings, $C^0$--converging to $\varphi_0\in \sy(X_{\Omega_1},X_{\Omega_2})\}$. Let $\{J_i\in \mathcal{J}(\widehat{W}_{\varphi_{i}})\}_{i \geq 1}$ be a sequence of compatible almost complex structures converging to $J_0\in \mathcal{J}(\widehat{W}_{\varphi_{0}})$. Let $u_i\in\mathcal{M}_{J_i}(\varphi_i)$. Then the sequence $\{u_i\}_{i\geq 1}$ cannot converge in the sense of \cite{bourgeois2003compactness} to a $J_0$--holomorphic building with more than one level.
\end{proposition}
\begin{proof}
In general, if there exists a $J$--holomorphic curve from the orbit set $\alpha$ to the orbit set $\beta$, then $\mathcal{A}(\alpha)\geq \mathcal{A}(\beta)$. Assume that, in the limit, the cylinders $u_i$ break into a $J_0$--holomorphic building $u_0=(v_1,v_2,\dots, v_l)$. Assume that $\alpha_{l}$ is the orbit set at which the level $v_l$ has negative ends. Then $\mathcal{A}(\alpha_l)\in [a,c]$. Note first that $c$ is the lowest action of an orbit set in $\partial X_{\Omega_2}$. This means that $v_1$ lives in the cobordism level. Secondly, the assumption $c_1^{\mathrm{ECH}}(X_{\Omega_2})< c_2^{\mathrm{ECH}}(X_{\Omega_1})$  translates to
\[ c<\min(2a, \mathcal{A}(e_{1,1}), 2f_1(0))=\min(\mathcal{A}(\gamma_a^2), \mathcal{A}(e_{1,1}), \mathcal{A}(e_{1,0}^2)),\]
where $\gamma_a=e_{0,1}$, $e_{1,1}$, and $e_{1,0}$ are the Reeb orbits on $\partial X_{\Omega_1}$. Thirdly, for a small enough perturbation of $\partial X_{\Omega_1}$, we also have $c<\mathcal{A}(h_{1,1})$ since $\mathcal{A}(h_{1,1})$ is approximately $ \mathcal{A}(e_{1,1})$. Lastly, Lemma \ref{lemma:higherindexorbits} implies that all orbit sets $\alpha$ on $\partial X_{\Omega_1}$ with $I(\alpha)\geq 5$ satisfy $c<\mathcal{A}(\alpha)$. 

Using the classification by ECH index in Lemma \ref{lemma:indexclassification}, together with the action inequalities above, we conclude that the only orbit set through which the cylinders $u_i$ could hypothetically break is $\alpha=e_{0,1}$.  This means that the only broken building we still have to rule out is $u_0=(v_1,v_2)$, where $v_1$ is a Fredholm index $0$ cylinder from $\gamma_c$ to $e_{1,0}$  in the cobordism level and  $v_2$ is a Fredholm index $0$ cylinder from $e_{1,0}$ to $\gamma_a=e_{0,1}$ in the lower symplectization level. The nontrivial cylinder $v_2$ is a Fredholm index $0$ $J_0$--holomorphic cylinder in a symplectization, and so, by automatic transversality, it cannot appear.

\end{proof}

Proposition \ref{proposition:nobadbreaking} together with the automatic transversality from Lemma \ref{lemma:automatictransversality}, and SFT compactness, \cite[Theorem 10.2]{bourgeois2003compactness}, imply that $\mathcal{M}_{J}(\varphi)$ is a compact $0$--dimensional manifold, i.e. a finite set of points.

\section{Proof of main theorem}
\label{section:proof}

\subsection{Non-emptiness of moduli spaces}
\label{subsection:nonemptiness}

First, we prove the nonemptiness of $\mathcal{M}_{\widehat{J}}(\varphi_0)$ for the inclusion map $\varphi_0:X_{\Omega_1}\to X_{\Omega_2}$ and a certain compatible almost complex structure $\widehat{J}$. 
\begin{proposition} 
\label{proposition:nonempty}
 There exists $\widehat{J} \in \mathcal{J}(\widehat{W}_{\varphi_0})$ such that the moduli space $\mathcal{M}_{\widehat{J}}(\varphi_0)$ is nonempty.
\end{proposition}
\begin{proof}
We will construct a compatible almost complex structure $\widehat{J}$ that is invariant under the $S^1$--action by rotations in the $z_2$--plane and prove that an appropriate restriction of the $z_1$--plane is the $\widehat{J}$--holomorphic cylinder we are looking for. Our construction is similar to \cite[\S 5.2]{burkard2016first}. Whenever we say ``$S^1$--equivariant", we mean invariant under the $S^1$--action by rotations in the $z_2$--plane.

Recall that $\partial X_{\Omega_1}$ and $\partial X_{\Omega_2}$ are contact hypersurfaces in the compact symplectic cobordism $(W_{\varphi_0}, \omega_{\mathrm{std}}=d\lambda_{\mathrm{std}})$. Moreover, notice that they are $S^1$--equivariant. Using an $S^1$--equivariant version of the Moser trick, one can prove that there exist $S^1$--equivariant neighborhoods $N_1$ of $\partial X_{\Omega_1}$ and $N_2$ of $\partial X_{\Omega_2}$ in $W_{\varphi_0}$, and $S^1$--equivariant symplectomorphisms
\[ \psi_1: (N_{1}, \omega) \to ([0,\epsilon)\times \partial X_{\Omega_1}, d(e^s\lambda_1))\] 
and
\[ \psi_2: (N_{2}, \omega) \to ((-\epsilon,0]\times \partial X_{\Omega_2}, d(e^s\lambda_2)),\]
where $\lambda_i = \lambda_{\mathrm{std}}|_{\partial X_{\Omega_i}}$, and $s$ denotes the coordinate on $[0,\epsilon)$ and $(-\epsilon,0]$.

Choose almost complex structures $J_1$ on $([0,\frac{\epsilon}{3})\cup(\frac{2\epsilon}{3},\epsilon)) \times \partial X_{\Omega_1}$ and $J_2$ on $((-\epsilon,-\frac{2\epsilon}{3})\cup(-\frac{\epsilon}{3},0)) \times \partial X_{\Omega_2}$, that are $S^1$--equivariant and compatible with the cylindrical ends near the boundary of $W_{\varphi_0}$, and that pull back under $\psi_i$ to the standard complex structure on $\mathbb{C}^2$ near the interior of $W_{\varphi_0}$, i.e. $\psi_1^{*} (J_1|_{(\frac{2\epsilon}{3},\epsilon)\times \partial X_{\Omega_1}}) = i$ and $\psi_2^{*} (J_2|_{(-\epsilon,-\frac{2\epsilon}{3})\times \partial X_{\Omega_2}}) = i$. Define 
\begin{equation}
\label{equation:tildej}
\widehat{J}(p):=
\begin{cases}
\psi_1^{*} (J_1(\psi_1(p))), & p\in\psi_1^{-1}((0,\frac{\epsilon}{3})\cup(\frac{2\epsilon}{3},\epsilon)) \times \partial X_{\Omega_1}) \\
i, & p\in W_{\varphi_0}\setminus (N_1\cup N_2)\\
\psi_2^{*} (J_2(\psi_2(p))), & p\in\psi_2^{-1}((-\epsilon,-\frac{2\epsilon}{3})\cup(-\frac{\epsilon}{3},0) \times \partial X_{\Omega_2}).
\end{cases} 
\end{equation}
The compatibility of $\widehat{J}$ with the cylindrical ends near the boundary of the compact symplectic cobordism $W_{\varphi_0}$ makes it possible to extend $\widehat{J}$ to a compatible $S^1$--equivariant almost complex structure on the cylindrical ends of the completed symplectic cobordism $\widehat{W}_{\varphi_0}$. We still need to interpolate between the standard complex structure in the interior of $W_{\varphi_0}$ and the almost complex structure on the cylindrical ends.

Let $g(\cdot,\cdot):=\omega(\cdot, \widehat{J}\cdot)$ be the positive definite Riemannian metric defined by the compatibility of $\omega$ and $\widehat{J}$ and note that $g$ is $S^1$--equivariant. Extend the Riemannian metric $g$ to $W_{\varphi_0}$ and average the obtained extension over the $S^1$--action to obtain an $S^1$--equivariant Riemannian metric $\widehat{g}$ on $W_{\varphi_0}$. Note that $\widehat{g}=g$ wherever $g$ is defined since $g$ is $S^1$--equivariant. Define $\widehat{J}$ to be the unique compatible almost complex structure that satisfies $\widehat{g}(\cdot, \cdot)=\omega(\cdot, \widehat{J} \cdot)$ and note that that this definition extends the definition in (\ref{equation:tildej}), since $\widehat{g}=g$ wherever $g$ is defined. Note that since $\widehat{g}$ and $\omega_{\mathrm{std}}$ are $S^1$--equivariant, then $\widehat{J}$ is also $S^1$--equivariant.

Let $S:= W_{\varphi_0}\cap \{ z_1=0\}$. Note that $S$ is a closed annulus which we can complete by adding cylindrical ends to get 
\[ \widehat{S}:=(-\infty,0]\times \gamma_a \cup S \cup [0,\infty) \times \gamma_c.\] 

We will now show that $\widehat{J}$ being invariant under $S^1$--action in the $z_2$--plane implies that $\widehat{J}$ preserves the tangent space of $\widehat{S}$. Let $h_{\theta}(z_1,z_2):=(z_1,e^{i\theta}z_2)$, for $\theta\in[0,2 \pi]$. Knowing $\widehat{J}$ is invariant under the $S^1$--action in the $z_2$--plane implies that
\[ \widehat{J}_{h_{\theta}(p)} \circ d_ph_{\theta} = d_ph_{\theta} \circ \widehat{J}_{p}, \] 
for any $p\in W_{\varphi_0}$ and any $\theta\in [0,2 \pi]$. In the basis $\left \lbrace \frac{\partial}{\partial x_1}, \frac{\partial}{\partial y_1}, \frac{\partial}{\partial x_2},\frac{\partial}{\partial y_2} \right \rbrace$, this equality can be written in $2\times 2$ block matrix notation as,

\begin{equation}
\label{equation:invariantj}
\left(\begin{array}{cc}
A & B \\
C & D 
\end{array}\right)_{h_{\theta}(p)}
\left(\begin{array}{cc}
I & 0 \\
0 & R_{\theta} 
\end{array}\right) 
= 
\left(\begin{array}{cc}
I & 0 \\
0 & R_{\theta} 
\end{array}\right) 
\left(\begin{array}{cc}
A & B \\
C & D 
\end{array}\right)_{p},
\end{equation}
for any $p\in W_{\varphi_0}$ and any $\theta\in [0,2 \pi]$, and where $\widehat{J}_p=\left(\begin{array}{cc}
A & B \\
C & D 
\end{array}\right)_{p}$ 
is the almost complex structure in coordinates and 
$R_{\theta} = \left(\begin{array}{cc}
\cos \theta & -\sin \theta \\
\sin \theta &\cos \theta
\end{array}\right)$
is a rotation matrix. After carrying out the multiplications in (\ref{equation:invariantj}), we see that
\[ 
\left(\begin{array}{cc}
A_{h_{\theta}(p)} & B_{h_{\theta}(p)} R_{\theta}\\
C_{h_{\theta}(p)} & D_{h_{\theta}(p)} R_{\theta}
\end{array}\right) 
=
\left(\begin{array}{cc}
A_p & B_p \\
R_{\theta}C_p & R_{\theta} D_p 
\end{array}\right).
\]
Note that for $p=(z_1,0)$, $h_{\theta}(p)=p$, and so the above equality implies $B_p R_{\theta} = B_p$ for any $p\in S$ and $\theta \in [0, 2\pi]$. This implies $B_p=0$ and hence, $\widehat{J}$ preserves the tangent bundle of $S$. Moreover, by construction, $\widehat{J}$ preserves the tangent spaces on the cylindrical ends of $\widehat{S}$ and so $\widehat{J}$ preserves the tangent bundle of $\widehat{S}$.

Hence, $(\widehat{S},\widehat{J})$ is a Riemann surface which is diffeomorphic to a punctured plane. By the Uniformization theorem, $(\widehat{S},\widehat{J})$ is biholomorphically equivalent to either the punctured plane, the punctured disk, or an open annulus. Since $\widehat{J}$ is compatible with the infinite cylindrical ends of $\widehat{W}_{\varphi_0}$, $(\widehat{S},\widehat{J})$ must be biholomorphic to a punctured plane, and hence also biholomorphic to a cylinder. We conclude that there exists a $\widehat{J}$--holomorphic map $u:(\mathbb{R}\times S^1, j)\to (\widehat{W}_{\varphi_0},\widehat{J})$ with image $\widehat{S}$, and hence, $[u]\in \mathcal{M}_{\widehat{J}}(\varphi_0)$. 

Finally, note that the perturbation of the hypersurfaces $\partial X_{\Omega_i}$, for $i=1,2$, needed to make $\lambda_{\mathrm{std}}|_{\partial X_{\Omega_i}}$ nondegenerate, happens away from the $z_1$--plane and so the curve $[u]$ persists after the perturbation. 
\end{proof}

\begin{remark}
\label{remark:sameimage}
All the symplectic embeddings that form the loop considered in Theorem  \ref{theorem:maintheorem}  have the same image in $X_{\Omega_2}$, so $\widehat{W}_{\varphi_t} = \widehat{W}_{\varphi_0}$, for any $t\in[0,1]$. Hence the moduli space $\mathcal{M}_{\widehat{J}}(\varphi_t)$ contains the same $\widehat{J}$--holomorphic cylinders as $\mathcal{M}_{\widehat{J}}(\varphi_0)$. 
\end{remark}

\subsection{Counting the cylinders}
\label{subsection:uniqueness}
We next prove the uniqueness of the $J$--holomorphic cylinders using asymptotic analysis estimates. Let us begin by recalling the \textit{adjunction formula}:
\begin{lemma}
Let $u:\dot{\Sigma}\to X$ be a somewhere--injective $J$--holomorphic curve. Then $u$ has finitely many singularities, and 
\begin{equation}
\label{equation:adjunctionformula}
c_{\tau}(u) = \chi(u)+Q_{\tau}(u)+w_{\tau}(u)-2\delta(u)
\end{equation}
where $c_{\tau}(u)$ is the relative first Chern class as before (see \cite[\S 4.2]{hutchings2009embedded}), $\chi(u)$ is the Euler characteristic of the domain of $u$, $Q_{\tau}(u)$ is the relative self intersection number as before (see \cite[\S 4.2]{hutchings2009embedded}), $w_{\tau}(u)$ is the asymptotic writhe defined in \cite[\S 2.6]{hutchings2009embedded}, and $\delta(u)$ is a count of singularities of $u$ with positive integer weights. 
\end{lemma}
For a proof of this statement, see \cite[\S 3]{hutchings2002index}. Following the details in \cite[\S 2.6]{hutchings2009embedded}, we give an overview of the definition writhe, linking number, and winding number in this context, as they will become useful in the proof of Proposition \ref{proposition:onecylinder} below.

Let $\gamma$ be a simple Reeb orbit and let $k$ be a positive integer. A \textit{braid} with $k$ strands around $\gamma$ is an oriented link $\zeta$ contained in a tubular neighborhood $N$ of $\gamma$, such that the tubular neighborhood projection $\zeta \to \gamma$ is an orientation--preserving degree $k$ submersion.

Choose a symplectic trivialization $\tau$ over $\gamma$ and extend it to the tubular neighborhood $N$ of $\gamma$ to identify $N$ with $S^1\times \mathbb{D}$, such that the projection of $\zeta\in N$ to the $S^1$ factor is a submersion. Identify further $S^1\times \mathbb{D}$ with a solid torus in $\mathbb{R}^3$ by applying an orientation preserving diffeomorphism. We thus obtain an embedding $\phi_{\tau}:N\to\mathbb{R}^3$. We set up the identifications in such a way that $\phi_{\tau}(\zeta)$ is an oriented link in $\mathbb{R}^3$ with no vertical tangents. Hence, it has a well defined writhe by counting signed self--crossings in the projection to $\mathbb{R}^2\times\{0\}$. We use the sign convention where counterclockwise twists contribute positively to the writhe.

We define the \textit{writhe} of a braid $\zeta$ around $\gamma$, $w_{\tau}(\zeta)\in\mathbb{Z}$, to be the writhe of the oriented link $\phi_{\tau}(\zeta)$ in $\mathbb{R}^3$. Also if $\zeta$ and $\zeta^{\prime}$ are two disjoint braids around $\gamma$, define the \textit{linking number} of $\zeta$ and $\zeta^{\prime}$, $l_{\tau}(\zeta,\zeta^{\prime})\in\mathbb{Z}$, to be the linking number of the oriented links $\phi_{\tau}(\zeta)$ and $\phi_{\tau}(\zeta^{\prime})$ in $\mathbb{R}^3$. This latter quantity is defined as one half the signed count of crossings of the projections of the two links to $\mathbb{R}^2\times \{0\}$. Note that, if $\zeta$ and $\zeta^{\prime}$ are two disjoint braids around $\gamma$ then 
\[ w_{\tau}(\zeta\cup\zeta^{\prime})=w_{\tau}(\zeta)+w_{\tau}(\zeta^{\prime})+2l_{\tau}(\zeta,\zeta^{\prime}).\]

For a braid $\zeta$ around $\gamma$ that is disjoint from $\gamma$ we define the \textit{winding number} of $\zeta$ around $\gamma$ to be $\mathrm{wind}_{\tau}(\zeta):=l_{\tau}(\zeta,\gamma)$.

The following two lemmas explain how to bound the writhe and the winding number in terms of the Conley--Zehnder index. The formulation is adapted from \cite{hutchings2014cylindrical}. For more details, see also \cite{hutchings2002index}.

\begin{lemma}[{\cite[Lemma 3.2]{hutchings2014cylindrical}}]
\label{lemma:positivewrithe}
Let $\gamma$ be an embedded Reeb orbit and let $N$ be a tubular neighborhood around $\gamma$. Let $u:\dot{\Sigma}\to \mathbb{R}\times Y$ be a $J$--holomorphic with a positive end at $\gamma^d$ which is not part of a trivial cylinder or a multiply covered component and let $\zeta$ denote the intersection of this end with $\{ s\} \times Y$. If $s>>0$, then the following hold:
\begin{enumerate}
\item[a.] $\zeta$ is the graph in $N$ of a nonvanishing section of $\xi_{\gamma^d}$ and has well defined winding number $\mathrm{wind}_{\tau}(\zeta)$.
\item[b.] $\mathrm{wind}_{\tau}(\zeta) \leq \left\lfloor \frac{CZ_{\tau}(\gamma^d)}{2} \right\rfloor$.
\item[c.] If $J$ is generic, $CZ_{\tau}(\gamma^d)$ is odd, and $\mathrm{ind}(u)\leq 2$ then equality holds in (b).
\item[d.] $w_{\tau}(\zeta)\leq (d-1) \mathrm{wind}_{\tau}(\zeta)$.
\end{enumerate}
\end{lemma}

An equivalent statement holds for the asymptotic winding number and writhe at a negative cylindrical end of a $J$--holomorphic curve.

\begin{lemma}[{\cite[Lemma 3.4]{hutchings2014cylindrical}}]
\label{lemma:negativewrithe}
Let $\gamma$ be an embedded Reeb orbit and let $N$ be a tubular neighborhood around $\gamma$. Let $u:\dot{\Sigma}\to \mathbb{R}\times Y$ be a $J$--holomorphic with a negative end at $\gamma^d$ which is not part of a trivial cylinder or a multiply covered component and let $\zeta$ denote the intersection of this end with $\{ s\} \times Y$. If $s<<0$, then the following hold:
\begin{enumerate}
\item[a.] $\zeta$ is the graph in $N$ of a nonvanishing section of $\xi_{\gamma^d}$ and has well defined winding number $\mathrm{wind}_{\tau}(\zeta)$.
\item[b.] $\mathrm{wind}_{\tau}(\zeta) \geq \left\lceil \frac{CZ_{\tau}(\gamma^d)}{2} \right\rceil$.
\item[c.] If $J$ is generic, $CZ_{\tau}(\gamma^d)$ is odd, and $\mathrm{ind}(u)\leq 2$ then equality holds in (b).
\item[d.] $w_{\tau}(\zeta)\geq (d-1) \mathrm{wind}_{\tau}(\zeta)$.
\end{enumerate}
\end{lemma}

Fix a symplectic embedding $\varphi \in \sy(X_{\Omega_1},X_{\Omega_2})$ and fix an almost complex structure $J\in\mathcal{J}(\widehat{W}_{\varphi})$. 

\begin{proposition} 
\label{proposition:onecylinder}
If the moduli space $\mathcal{M}_J(\varphi)$  is nonempty, then it contains exactly one index zero cylinder.
\end{proposition}
\begin{proof}
 Assume there are two different cylinders, $u_1$ and $u_2$, in $\mathcal{M}_J(\varphi)$. For $s<<0$, $\zeta_a=(u_1\cup u_2)\cap (\{s\}\times \partial X_{\Omega_1})$ is a braid around $\gamma_a$ with two components, $\zeta^{a}_1$ and $\zeta^{a}_2$, each having one strand. For $s>>0$, $\zeta_c=(u_1\cup u_2)\cap (\{s\}\times \partial X_{\Omega_2})$ is a braid around $\gamma_c$ with two components, $\zeta^{c}_1$ and $\zeta^{c}_2$, each with one strand.
Lemma \ref{lemma:positivewrithe} implies
\[ \mathrm{wind}_{\tau}(\zeta^{c}_i) \leq \left\lfloor \frac{CZ_{\tau}(\gamma_c)}{2} \right\rfloor = \left\lfloor \frac{1}{2} \right\rfloor = 0.\]
Similarly, Lemma \ref{lemma:negativewrithe} implies
\[ \mathrm{wind}_{\tau}(\zeta^{a}_i) \geq \left\lceil \frac{CZ_{\tau}(\gamma_a)}{2} \right\rceil = \left\lceil \frac{1}{2} \right\rceil = 1.\]
The linking numbers of the different strands of the two braids are given by $l_{\tau}(\zeta^a_1,\zeta^a_2)  = \mathrm{wind}(\zeta^a_2)$  and  $l_{\tau}(\zeta^c_1,\zeta^c_2) = \mathrm{wind}(\zeta^c_2)$. See \cite[Lemma 4.17]{hutchings2009embedded} for details.
This means 
\begin{align*}
\hspace{1.5cm} w_{\tau}(\zeta_a) & = w_{\tau}(\zeta^a_1\cup\zeta^a_2) = w_{\tau}(\zeta^a_1) + w_{\tau}(\zeta^a_2) + 2\cdot l_{\tau}(\zeta^a_1,\zeta^a_2)  \\
& = 0 + 0 + 2 \cdot \mathrm{wind}(\zeta^a_2) \geq 2 
\end{align*}
and \begin{align*}
\hspace{1.5cm} w_{\tau}(\zeta_c) & = w_{\tau}(\zeta^c_1\cup\zeta^c_2) = w_{\tau}(\zeta^c_1) + w_{\tau}(\zeta^c_2) + 2\cdot l_{\tau}(\zeta^c_1,\zeta^c_2)  \\
& = 0 + 0 + 2 \cdot \mathrm{wind}(\zeta^c_2) \leq 0. 
\end{align*}
Hence
\[ w_{\tau}(u_1\cup u_2) = w_{\tau}(\zeta_c)-w_{\tau}(\zeta_a)\leq -2.\]
Since $c_{\tau}(u_1\cup u_2)=Q_{\tau}(u_1\cup u_2) = 0$, the relative adjunction formula recalled in (\ref{equation:adjunctionformula}) applied to $u_1\cup u_2$ gives
\[ 0 = 0+0+w_{\tau}(u_1\cup u_2)-2\delta(u_1\cup u_2).\]
This is a contradiction since $w_{\tau}(u_1\cup u_2)\leq -2$ and $\delta(u_1\cup u_2)\geq 0$. 
\end{proof}

\subsection{Final steps of the proof}
\label{subsection:proof}
We have all the details needed to complete the proof of Theorem \ref{theorem:maintheorem}. Assume that the loop $\{\varphi_t\}_{t\in[0,1]}$ is contractible in $\sy(X_{\Omega_1},X_{\Omega_2})$. This means there exists a $2$--parameter family $\{ \varphi_z \}_{z\in\mathbb{D}}\subset \sy(X_{\Omega_1},X_{\Omega_2})$, parametrized by the unit disk $\mathbb{D}$, such that $\{ \varphi_z \}_{z\in\partial \mathbb{D}} = \{ \varphi_t \}_{t\in[0,1]}$. The family of embeddings $\{\varphi_z\}_{z\in\mathbb{D}}$ generates a $2$--parameter family of completed symplectic cobordisms $\{\widehat{W}_{\varphi_z}\}_{z\in\mathbb{D}}$. Let $\mathfrak{J}=\{ J_z\}_{z\in\mathbb{D}}$  be a generic $2$--parameter family of compatible almost complex structures such that $J_z\in\mathcal{J}(\widehat{W}_{\varphi_z})$ for every $z\in\mathbb{D}$ and $J_z=\widehat{J}$ for every $z\in\partial\mathbb{D}$, where $\widehat{J}$ is the almost complex structure constructed in Proposition \ref{proposition:nonempty}. Remark \ref{remark:sameimage} provides an explanation as to why we can choose the same almost complex structure $\widehat{J}$ for all $z\in\partial\mathbb{D}$.

Consider the moduli space
\[\mathcal{M}_{\mathfrak{J}}:=\left\lbrace (z,u_z) \; \middle| \; z\in\mathbb{D}, \; u_z\in\mathcal{M}_{J_z} (\varphi_z) \right\rbrace. \]

\begin{claim} $\mathcal{M}_{\mathfrak{J}}$ is homeomorphic to the closed disk $\mathbb{D}$.
\end{claim}
\begin{proof}
By the parametric regularity theorem, \cite[Theorem. 7.2 \& Remark 7.4]{wendl2016lectures}, for a generic choice of $2$--parameter family of compatible complex structures $\mathfrak{J}$, the moduli space $\mathcal{M}_{\mathfrak{J}}$ is a $2$--dimensional manifold that is cut out transversely. The holomorphic curves in $\mathcal{M}_{\mathfrak{J}}$ have fixed asymptotics and so, by the SFT compactness result presented in \cite[Theorem 10.2]{bourgeois2003compactness}, there exists a compactification of $\mathcal{M}_{\mathfrak{J}}$ with broken holomorphic buildings. Proposition \ref{proposition:nobadbreaking} implies that, under the assumptions made in the hypothesis of Theorem \ref{theorem:maintheorem}, no such breaking is possible and so, $\mathcal{M}_{\mathfrak{J}}$ is already compact.

The automatic transversality result presented in Lemma \ref{lemma:automatictransversality}, together with the nonemptiness result proved in Proposition \ref{proposition:nonempty} and the uniqueness result proved in Proposition \ref{proposition:onecylinder}, implies that $\mathcal{M}_{\mathfrak{J}}$ contains exactly one cylinder above each parameter $z\in\partial\mathbb{D}$ and at most one cylinder above each parameter $z\in \mathrm{int}\mathbb{D}$. Given that the moduli space $\mathcal{M}_{\mathfrak{J}}$ is compact, it must contain exactly one cylinder above every parameter $z\in\mathbb{D}$ and so we can conclude that $\mathcal{M}_{\mathfrak{J}}$ is homeomorphic to the disk $\mathbb{D}$.
\end{proof}

Let $\gamma_c:\mathbb{R}/c\mathbb{Z}\to \partial X_{\Omega_2}$ be the parametrization of $\gamma_c$ such that $p=\gamma_c(0)=\left(\sqrt{\frac{c}{\pi}},0\right)\in\mathbb{C}^2$.  There exists a unique representative $u_z:\mathbb{R}\times S^1 \to \widehat{W}_{\varphi_z}$ of the unique class in $\mathcal{M}_{J_z} (\varphi_z)$ such that $\lim_{s\to\infty} u_z(s,0) = p$. Define $p_z:=\lim_{s\to-\infty} u_z(s,0)$. 
This construction induces a well defined composition of maps
\[
\begin{array}{ccccccc}
S^1 & \to & \sy(X_{\Omega_1},X_{\Omega_2})  & \to & \mathcal{M}_{\mathfrak{J}} & \to & \gamma_a \simeq S^1 \\
t & \mapsto & \varphi_t=\varphi_z & \mapsto & (z,[u_z]) & \mapsto & p_z.
\end{array}
\]
\begin{claim}
The above composition is a degree $-1$ circle map.
\end{claim}
\begin{proof}
Remark \ref{remark:sameimage} explains why for any two parameters $z,w\in\partial \mathbb{D}$, the moduli spaces  $\mathcal{M}_{J_z} (\varphi_z)$ and $\mathcal{M}_{J_w} (\varphi_w)$ are the same. Moreover, note that the choice of fixed asymptotics, $\lim_{s\to\infty} u_z(s,0) = p = \lim_{s\to\infty} u_w(s,0) $, implies that the representatives $u_z$ and $u_w$ are also the same. Hence, we can easily trace the movement of the point $p_z$ on the orbit $\gamma_a$ as $z$ goes around the boundary of the parameter space. 

Recall that the image of $X_{\Omega_1}$ under the loop of symplectic $\{ \varphi_t\}_{t\in[0,1]}$ does a counterclockwise $2\pi$ rotation in the $z_1$--plane, which rotates the orbit $\gamma_a$, followed by a clockwise $2\pi$ rotation in the $z_2$--plane, which does not rotate the orbit $\gamma_a$. Let $q:=p_{1}$ be the point on $\gamma_a$ corresponding to the parameter $1\in\mathbb{D}$. Then 
\[p_{e^{2\pi i t}} =
\begin{cases} e^{-4\pi i t } q, & t \in \left[ 0,\frac{1}{2}\right] \\
q, & t \in \left( \frac{1}{2}, 1 \right],
\end{cases}\]
and so the above composition is a degree $-1$ circle map.
\end{proof}
This last claim provides us with a contradiction, given that a degree $-1$ circle map cannot factor through the disk $\mathcal{M}_{\mathfrak{J}}\simeq \mathbb{D}$.

\printbibliography

\end{document}